\documentclass[letterpaper]{ieeeconf}
\IEEEoverridecommandlockouts                              
\usepackage[utf8]{inputenc}
\usepackage{textcomp}

\usepackage{graphicx}
\usepackage{url}
\usepackage[table]{xcolor}
\usepackage{amsmath}
\usepackage[version=4]{mhchem}
\usepackage{siunitx}
\usepackage{graphicx}
\usepackage{longtable,tabularx}
\usepackage[style=ieee,doi=true]{biblatex}
\usepackage{tikz}
\usepackage{comment}
\usetikzlibrary{positioning, arrows.meta}
\AtBeginBibliography{\small}

\usepackage{amsmath} 
\usepackage{amsfonts}
\usepackage{amsthm}

\newtheorem{theorem}{Theorem}

\newtheorem{proposition}{Proposition}[section]
\theoremstyle{definition}

\newtheorem{remark}{Remark}

\usepackage{float} 
\usepackage[skip=0em,font=footnotesize]{caption}

\usepackage[margin = 1in]{geometry}

\usepackage{tikz}
\usetikzlibrary{shapes,arrows,calc,positioning}
\tikzstyle{bigblock} = [draw, fill=blue!20, rectangle, 
    minimum height=6em, minimum width=8em]
\tikzstyle{medblock} = [draw, fill=blue!20, rectangle, 
    minimum height=4em, minimum width=4em]    
\tikzstyle{mux} = [draw, fill=black!20, rectangle, 
    minimum height=5em, minimum width=0.1em]    
\tikzstyle{smallblock} = [draw, fill=blue!20, rectangle, 
    minimum height=2em, minimum width=3em]
    
\tikzstyle{data_block} = [draw, fill=green!20, rectangle, 
    minimum height=2em, minimum width=3em]
\tikzstyle{ops_block} = [draw, fill=blue!20, rectangle, 
    minimum height=2em, minimum width=3em]    
\tikzstyle{est_block} = [draw, fill=red!20, rectangle, 
    minimum height=2em, minimum width=3em]    
    
\tikzstyle{sum} = [draw, fill=blue!20, circle, node distance=1cm,minimum height=0.5cm]
\tikzstyle{signal} = [coordinate]
\tikzstyle{pinstyle} = [pin edge={to-,thin,black}]
\tikzstyle{block} = [draw, fill=blue!20, rectangle, 
    minimum height=3em, minimum width=9em]
\tikzstyle{blockS} = [draw, fill=blue!20, rectangle, 
    minimum height=3em, minimum width=4em]    
\tikzstyle{input} = [coordinate]
\tikzstyle{output} = [coordinate]
\usetikzlibrary{matrix}

\usetikzlibrary{positioning}
\usetikzlibrary{math}

\usepackage{hyperref}
\usepackage{xcolor}
\hypersetup{
    colorlinks,
    linkcolor={blue!100!black},
    citecolor={blue!50!black},
    urlcolor={blue!80!black}
}

\newcommand{\bc}{\begin{center}}
\newcommand{\ec}{\end{center}}
\newcommand{\benum}{\begin{enumerate}}
\newcommand{\eenum}{\end{enumerate}}
\newcommand{\nn}{\nonumber}
\newcommand{\matl}{\left[ \begin{array}}
\newcommand{\matr}{\end{array} \right]}

\renewcommand{\matl}{\begin{bmatrix}}
\renewcommand{\matr}{\end{bmatrix}}

\newcommand{\matls}{\left[ \begin{smallmatrix}}
\newcommand{\matrs}{\end{smallmatrix} \right]}
\newcommand{\isdef}{\stackrel{\triangle}{=}}

\newcommand{\rmI}{{\rm I}}

\newcommand{\rmd}{{\rm d}}

\newcommand{\rms}{{\rm s}}

\newcommand{\BBR}{{\mathbb R}}

\newcommand{\SA}{{\mathcal A}}

\newcommand{\SC}{{\mathcal C}}

\newcommand{\SF}{{\mathcal F}}

\newcommand{\SSS}{{\mathcal S}}

\usepackage{titlesec}
\titlespacing{\section}{0pt}{6pt}{3pt}
\titlespacing{\subsection}{0pt}{5pt}{2pt}
\titlespacing{\subsubsection}{0pt}{4pt}{2pt}

\usepackage[colorinlistoftodos,textwidth=1 in,textsize=footnotesize,disable]{todonotes}

\title{State-Constrained Control of Discrete-Time Nonlinear Systems \\ via Constraint Lifting}
\author{
Jhon Manuel Portella Delgado
and
Ankit Goel
\thanks{Jhon Manuel Portella Delgado and Ankit Goel are with the Department of Mechanical Engineering, University of Maryland, Baltimore County, 1000 Hilltop Circle, Baltimore, MD 21250. {\tt\small jportella, ankgoel@umbc.edu}}%
\thanks{Jhon Manuel Portella Delgado is now with the Department of Aerospace Engineering, University of Michigan, Ann Arbor, MI 48109, USA.}
}

\begin{document}

\maketitle

\listoftodos

\todo{remove to do list}


\begin{abstract}
This paper presents a constraint-enforcing control framework for a class of discrete-time strict-feedback nonlinear systems.
The objective is to guarantee closed-loop stability while ensuring forward invariance of a prescribed safe set defined by state constraints.
The proposed approach transforms the constrained control problem into an equivalent unconstrained one through smooth constraint-lifting mappings constructed using strictly increasing sigmoid functions. 
Controller synthesis is then performed in the lifted coordinates, enabling recursive backstepping design while preserving the admissibility of the constrained states.
Conditions on the controller gains are derived to guarantee both asymptotic stability of the closed-loop system and forward invariance of the admissible domain of the lifting functions.
The analysis also establishes a conditional deadbeat property for the second backstepping step once the system trajectory enters a region in which the lifting-domain admissibility conditions are satisfied. 
Numerical simulations of a constrained double-integrator system demonstrate the effectiveness of the proposed method in enforcing state constraints while tracking a reference command.

\end{abstract}
\textit{\bf keywords:} 
Constraint-lifting control, State-constrained control, Discrete-time nonlinear systems, Forward invariance, Backstepping control

\todo{fix g2}
\todo{read intro}
\section{Introduction}

The enforcement of state constraints in discrete-time nonlinear systems is closely tied to the problem of forward invariance of a prescribed safe set.
This problem arises naturally in digitally implemented control systems, where safety, actuator limits, and admissibility conditions must be satisfied at every sampling instant.
Set-invariance methods provide the underlying theoretical framework for this problem, while constrained-control architectures, such as reference governors, offer practical mechanisms to enforce constraints without redesigning the nominal stabilizing controller \cite{Blanchini1999, Gilbert1995, Garone2017}.

A widely used strategy in discrete-time constrained control is to augment a nominal controller with a supervisory law, such as a reference governor, that modifies the command whenever constraint violation is predicted \cite{Gilbert1995, Garone2017}.
These methods are effective and well-established, but they serve as add-on schemes rather than yielding an explicit stabilizing control law for the constrained plant itself. 
More recently, barrier-based methods have also been extended to discrete-time systems through discrete control barrier functions, which enforce forward invariance by imposing a safety inequality at each time step \cite{Agrawal2017}. 
However, these methods are typically implemented through an optimization problem, and in the discrete-time setting, the resulting program is not necessarily convex \cite{Agrawal2017}.
For nonlinear systems, this may increase computational complexity and complicate the stability analysis of the closed-loop system.

Another approach addresses constraints via nonlinear coordinate transformations that map a bounded state space to an unconstrained one, thereby enabling the use of classical nonlinear control tools on the transformed dynamics \cite{Guo2014}. 
Although this idea is attractive, conventional transformation-based designs and barrier Lyapunov formulations may suffer from numerical ill-conditioning near the constraint boundary, where the transformed coordinates or their derivatives become large \cite{Tee2009, Guo2014}. 
These difficulties become more pronounced when several constrained states are handled recursively.

Motivated by these limitations, this paper develops a discrete-time constraint-lifting framework for a class of nonlinear systems.
The central idea is to map constrained states into an unconstrained coordinate system via smooth constraint-lifting transformations constructed from strictly increasing sigmoid functions. 
Controller synthesis is then carried out in the lifted coordinates using a recursive design, while the inverse transformation guarantees that the original state constraints are satisfied.

The contributions of this paper are as follows. 
First, a constraint-lifting control framework is developed for discrete-time strict-feedback systems that transforms the constrained control problem into an equivalent unconstrained one. 
Second, conditions on the controller gains are derived to guarantee both asymptotic stability of the closed-loop system and forward invariance of the admissible domain of the lifting transformations. 
Third, the analysis establishes a conditional deadbeat property for the second backstepping step once the closed-loop trajectory enters a region in which the admissibility conditions are satisfied.

The proposed approach yields an explicit control law, avoids online optimization, and ensures constraint satisfaction and closed-loop stability. 
Numerical simulations illustrate the effectiveness of the proposed framework on a constrained double-integrator example.

\section{Problem Formulation}
Consider the system
\begin{align}
    x_{1_{k+1}}
        &=
            f_1(x_{1_k})
            +
            g_1(x_{1_k})
            x_{2_k}
    \label{eq:x1_dot}
    \\
    x_{2_{k+1}}
        &=
            f_2(x_{1_k},x_{2_k})
            +
            g_2(x_{1_k},x_{2_k})
            u_{k},
    \label{eq:x2_dot}
\end{align}
where $x_{1_k},$ $x_{2_k} \in \BBR,$ are the states, $u_k \in \BBR$ is the control, and the functions $f_1 : \BBR \to \BBR,$ $g_1 : \BBR \to \mathbb{R},$ $f_2 : \BBR \times \BBR \to \BBR,$ and $g_2 : \BBR \times \BBR \to \mathbb{R}.$
The goal is to design a feedback controller $u_k$ such that $x_{1_k}$ tracks a desired command $x_{1\rmd_k}$ while $|x_{1_k}| < \overline{x}_1,$ and $|x_{2_k}| < \overline{x}_2.$
In this work, we consider the scalar case to simplify the exposition and focus on the key ideas of the proposed constraint-lifting framework. 
The extension to multi-dimensional systems is conceptually similar and will be developed in future work.

\section{Constraint Lifting}
\label{sec:Constraint_Lifting}

The states to be constrained are \textit{lifted} using a constraint-lifting transformation to obtain an unconstrained representation in a lifted state space.
The sigmoid functions and their inverses to lift the constraints are described in detail in \cite{portella2025constrained}.

First, the states $(x_1,x_2)$ are mapped to intermediate variables $(\chi_1,\chi_2)$ such that the bounds on the components of $x_1$ and $x_2$ are normalized to $\pm1$. 
Consequently, whenever $(x_1,x_2)\in\SSS$, each component of $(\chi_1,\chi_2)$ lies in $(-1,1)$, that is, $(\chi_1,\chi_2)\in\SC$.
The normalized variables $(\chi_1,\chi_2)$ are then mapped to $(z_1,z_2)\in\BBR^n$ using a constraint-lifting function. 
Controller synthesis is performed in the lifted $z$-coordinates for the transformed state $x_{1\rmd}$.
Finally, a transformation from $(z_1,z_2)$ to $(\zeta_1,\zeta_2)$ is introduced to simplify both the inverse mapping from $z$ to $x$ and the subsequent stability analysis.

The sequence of transformations from $x$ to $z$, together with the intermediate variables $\chi$ and $\zeta$, is shown in Figure \ref{fig:state_transformation}.

\begin{figure}[!ht]
\centering
    {%
    \begin{tikzpicture}[>={stealth'}, line width = 0.25mm]

    \node [input, name=ref]{};
    
    \node [smallblock, fill=green!20, rounded corners, right = 0.5cm of ref , 
           minimum height = 0.6cm, minimum width = 0.7cm] 
           (controller) {$\dfrac{x}{\overline{x}}$};
    
    \node [smallblock, rounded corners, right = 2cm of controller, 
           minimum height = 0.6cm , minimum width = 0.7cm] 
           (integrator)            {$\overline x\,\phi(\chi)$};
    
    \node [smallblock, fill=green!20, rounded corners, below = 0.5cm of integrator, 
           minimum height = 0.6cm , minimum width = 0.7cm] 
           (system) {$\dfrac{z}{\overline{x}}$};
           
    \node [smallblock, rounded corners, below = 0.5cm of controller, 
           minimum height = 0.6cm , minimum width = 0.7cm] 
           (integrator2)            {$\overline x\,\psi(\zeta)$};

    \node [output, right = 1.0cm of system] (output) {};
    
    
    \node [input, left = 1.0cm of controller] (reference) {};
    
    \draw [->] (controller) -- node [above] {$\chi\in \SC$} (integrator);
    \draw [->] (integrator.0) -- +(1,0) node [above, xshift = -1em] {$z\in\BBR^n$} |-  (system.0);
    \draw [->] (system.180)  -- node [above] {$\zeta \in \BBR^n$} (integrator2.0);
    \draw [->] (integrator2.180) -- +(-1,0) |- node [above, xshift = 1em] {$x \in \SSS$} (controller.180);
    \end{tikzpicture}
    }  
    \caption{
        State transformations with the constraint-lifting function and its inverse used in the control synthesis.
    }
    \label{fig:state_transformation}
\end{figure}
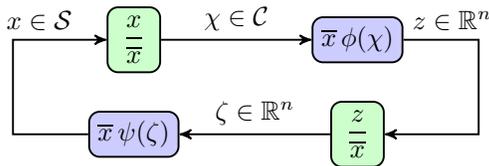

Define the normalized states
\begin{align}
    \chi_{1_k}
        &\isdef
            \dfrac{x_{1_k}}{\overline{x}_1},
    \qquad
    \chi_{2_k}
        \isdef
            \dfrac{x_{2_k}}{\overline{x}_2}.
\end{align}
By construction, the constraint boundaries are mapped to $\pm1$. 
Hence, if $(x_1,x_2)\in\SSS$, then $(\chi_1,\chi_2)\in\SC$.
Next, define the lifted coordinates
\begin{align}
    z_{1_k} 
        &\isdef
            \overline{x}_1\,\phi_1(\chi_{1_k}),
    \qquad
    z_{2_k} 
        \isdef
            \overline{x}_2\,\phi_2(\chi_{2_k}),
    \label{eq:x2_to_z2}
\end{align}
where $\phi_i:\SC\to\BBR$ is the constraint-lifting function.
To simplify controller synthesis, we define the auxiliary variables
\begin{align}
    \zeta_{1_k}
        &\isdef
            \dfrac{z_{1_k}}{\overline{x}_1},
    \qquad
    \zeta_{2_k}
        \isdef
            \dfrac{z_{2_k}}{\overline{x}_2}.
    \label{eq:zeta2_def}
\end{align}
The inverse transformation is then defined by
\begin{align}
    x_{1_k} 
        &=
            \overline{x}_1\,\psi_1(\zeta_{1_k}),
    \qquad
    x_{2_k} 
        =
            \overline{x}_2\,\psi_2(\zeta_{2_k}),
    \label{eq:z2_to_x2}
\end{align}
where $\psi_i:\BBR\to\SSS$.

Under these transformations, the dynamics \eqref{eq:x1_dot}--\eqref{eq:x2_dot} expressed in the lifted coordinates $z$ become
\begin{align}
    z_{1_{k+1}}
        &=
            \overline{x}_1
            \phi_1
            \bigl(
                \SF_1(z_{1_k},z_{2_k})
            \bigr),
    \label{eq:z1_dot}
    \\
    z_{2_{k+1}}
        &=
            \overline{x}_2
            \phi_2
            \bigl(
                \SF_2(z_{1_k},z_{2_k},u_k)
            \bigr),
    \label{eq:z2_dot}
\end{align}
where
{\small{
\begin{align}
    \SF_1(z_{1_k},z_{2_k})
        &\isdef
            \dfrac{
                f_1\bigl(\overline{x}_1\psi_1(\zeta_{1_k})\bigr)
            }
            {\overline{x}_1}
            \nn \\ &\quad 
            +
            \dfrac{\overline{x}_2}{\overline{x}_1}
            g_1\bigl(\overline{x}_1\psi_1(\zeta_{1_k})\bigr)
            \psi_2(\zeta_{2_k}),
    \\
    \SF_2(z_{1_k},z_{2_k},u_k)
        &\isdef
            \dfrac{
                f_2\bigl(
                    \overline{x}_1\psi_1(\zeta_{1_k}),
                    \overline{x}_2\psi_2(\zeta_{2_k})
                \bigr)
            }
            {\overline{x}_2}
            \nn \\ &\quad 
            +
            \dfrac{
                g_2\bigl(
                    \overline{x}_1\psi_1(\zeta_{1_k}),
                    \overline{x}_2\psi_2(\zeta_{2_k})
                \bigr)
            }
            {\overline{x}_2}
            u_k.
\end{align}
}}

\section{Controller Synthesis}
\label{sec:Controller_design}

Let $x_{1\rmd_k}\in\SSS$ denote the desired value of $x_1$. Define the corresponding normalized and lifted states
\begin{align}
    \chi_{1\rmd_k}
        &\isdef
            \frac{x_{1\rmd_k}}{\overline{x}_1},
    \ 
    z_{1\rmd_k}
        \isdef
            \overline{x}_1\,\phi_1(\chi_{1\rmd_k}),
    \ 
    \zeta_{1\rmd_k}
        \isdef
            \frac{z_{1\rmd_k}}{\overline{x}_1}.       
    \label{eq:z1d}
\end{align}
and define the tracking error
\begin{align}
    e_{1_k}
        &\isdef
            \psi_1
            \left(
                \zeta_{1_k}
            \right)
            -
            \psi_1
            \left(
                \zeta_{1\rmd_k}
            \right).
    \label{eq:e1_def}
\end{align}

\subsection{\texorpdfstring{$e_1$}{e1} stabilization}

Consider the Lyapunov function
    $V_{1_k}
        \isdef
            \frac{1}{2}e_{1_k}^2 .$
Define $\Delta V_{1_k} \isdef V_{1_k}(e_{1_{k+1}})-V_{1_k}(e_{1_k})$. Then
\begin{align}
    \Delta V_{1_k}
        &=
            \frac{1}{2}
            \left(
                \SF_1(z_{1_k},z_{2_k})
                -
                \psi_1
                \!\left(
                    \zeta_{1\rmd_{k+1}}  
                \right)
            \right)^2
            -
            \frac{1}{2}e_{1_k}^2 .
\end{align}

If
\begin{align}
    \SF_1(z_{1_k},z_{2_k})
        -
        \psi_1
        \!\left(
            \zeta_{1\rmd_{k+1}}  
        \right)
        =
        \rho_1 e_{1_k},
    \label{eq:first_BS_Step}
\end{align}
where $|\rho_1|<1$, then
    $\Delta V_{1_k}
        =
        (\rho_1^2-1)V_{1_k}.$
%
However, the state $z_{2_k}$ is not the control input, and thus \eqref{eq:first_BS_Step} cannot be arbitrarily enforced. 

Instead, we define
\begin{align}
    z_{2\rmd_k}
        &\isdef
            \SF_1^\rmI
            \left(
                z_{1_k},
                \rho_1 e_{1_k}
                +
                \psi_1
                \!\left(
                    \zeta_{1\rmd_{k+1}}  
                \right)
            \right),
    \label{eq:z2dk}
\end{align}
where $\SF_1^\rmI(z_1,\cdot)$ denotes the inverse mapping satisfying
    $z_2=\SF_1^\rmI(z_1,y)
    \Longleftrightarrow 
\SF_1(z_1,z_2)=y .$
Using \eqref{eq:z2dk}, define
\begin{align}
    e_{2_k}
        &\isdef
            \psi_2
            \!\left(
                \zeta_{2_k}
            \right)
            -
            \psi_2
            \!\left(
                \zeta_{2\rmd_k}
            \right),
\end{align}
where $\zeta_{2\rmd_k} \isdef {z_{2\rmd_k}}/{\overline{x}_2}.$
Next, we construct a control law that ensures that $e_{2_k} \to 0$ and thus $z_{2_k} - z_{2\rmd_k} \to 0.$

\subsection{\texorpdfstring{$e_2$}{e2} stabilization}

Consider
    $V_{2_k}
        \isdef
            \frac{1}{2}e_{2_k}^2 .$
Define $\Delta V_{2_k}\isdef V_{2_k}(e_{2_{k+1}})-V_{2_k}(e_{2_k})$. Then
\begin{align}
    \Delta V_{2_k}
        &=
            \frac{1}{2}
            \left(
                \SF_2(z_{1_k},z_{2_k},u_k)
                -
                \psi_2
                \!\left(
                    \zeta_{2\rmd_{k+1}}   
                \right)
            \right)^2
            -
            \frac{1}{2}e_{2_k}^2 .
\end{align}

If
\begin{align}
    \SF_2(z_{1_k},z_{2_k},u_k)
        -
        \psi_2
        \!\left(
            \zeta_{2\rmd_{k+1}}  
        \right)
        =
        \rho_{2_k} e_{2_k},
    \label{eq:second_BS_Step}
\end{align}
where $|\rho_{2_k}|<1$, then
    $\Delta V_{2_k}
        =
        (\rho_{2_k}^2-1)V_{2_k}.$
Thus, we consider the control law
\begin{align}
    u_k
        =
            \SF_2^\rmI
            \Bigg(
                z_{1_k},
                z_{2_k},
                \rho_{2_k}e_{2_k}
                +
                \psi_2
                \!\left(
                    \zeta_{2\rmd_{k+1}}   
                \right)
            \Bigg),
\label{eq:u_k_constraint_lifting}
\end{align}
where $\SF_2^\rmI(z_1,z_2,\cdot)$ denotes the inverse mapping satisfying
    $u=\SF_2^\rmI(z_1,z_2,y)
    \Longleftrightarrow 
    \SF_2(z_1,z_2,u)=y .$

\begin{remark}
Note that 
\begin{align}
    e_{2_{k+1}}
        &=
            \rho_{2_k} e_{2_k}.
    \label{eq:e2_dynamics_theorem}
\end{align}
Consequently, if for some $k_\rms>0,$ $\rho_{2_k}=0,$ then, for all $k\ge k_\rms,$ $e_{2_k} =0.$
\end{remark}

\subsection{Stability Analysis}
Define the composite Lyapunov difference
\begin{align}
    \Delta V_k
        &\isdef
            \Delta V_{1_k}
            +
            \Delta V_{2_k}
        \\
        &=
            (\rho_1^2-1)V_{1_k}
            +
            (\rho_{2_k}^2-1)V_{2_k}.
\end{align}
If $|\rho_1|<1$ and $|\rho_{2_k}|<1$, then, the origin of the closed-loop system \eqref{eq:z1_dot}--\eqref{eq:z2_dot} under the controller \eqref{eq:u_k_constraint_lifting} is asymptotically stable \cite{khalil2002nonlinear}.

\section{Admissibility and Forward Invariance of the Constraint-Lifting Transformations}
\label{sec:Admisibility_constraint_leifting}
Admissibility of the constraint-lifting maps $\phi_1$ and $\phi_2$ requires that their arguments remain within the domain $(-1,1)$. 
From \eqref{eq:z1_dot}-\eqref{eq:z2_dot}, this condition is satisfied if, for all $k \ge 0,$
\begin{align}
    |\SF_1(z_{1_k},z_{2_k})| &< 1,
    \quad
    |\SF_2(z_{1_k},z_{2_k},u_k)| < 1.
\end{align}
These conditions ensure that the arguments of $\phi_1$ and $\phi_2$ remain in $(-1,1)$, so that the transformations \eqref{eq:z1_dot}--\eqref{eq:z2_dot} are well defined along the closed-loop trajectories.
Define the admissible sets
\begin{align}
    \SA_1
        &\isdef
            \left\{
                (z_1,z_2)\in\BBR^2
                :
                |\SF_1(z_1,z_2)| < 1
            \right\}, 
    \\
    \SA_2
        &\isdef
            \left\{
                (z_1,z_2,u)\in\BBR^3
                :
                |\SF_2(z_1,z_2,u)| < 1
            \right\}.
\end{align}

\subsection{Forward Invariance of the Admissible Domain of $\phi_1$}
The following result establishes forward invariance of the admissible domain of the lifting function $\phi_1$. 
To simplify the presentation, the theorem is stated for $\rho_1=0$. 
The general case $|\rho_1|<1$ introduces additional conditions on the initial state and leads to a considerably longer proof.

\begin{theorem}[Forward Invariance of the Admissible Domain of $\phi_1$]
\label{thm:forward_invariance_phi_1}
Consider the transformed dynamics \eqref{eq:z1_dot} and the first contraction condition \eqref{eq:first_BS_Step}.
Assume that the initial state $x_{0} \in \SSS$ and the corresponding lifted state $(z_{1_0}, z_{2_0}) \in \SA_1$.
Assume that for all $k \geq 0,$  $|x_{1\rmd_k}| < \overline{x}_1.$
Let $\rho_1 = 0.$
Then, the admissible domain of the lifting function $\phi_1$ is forward invariant along the closed-loop trajectories, that is, for all $k \geq 0,$
\begin{align}
    \left|
        \SF_1(z_{1_k},z_{2_k})
    \right|
    <
    1.
    \label{eq:phi1_domain_requirement}
\end{align}
\end{theorem}

\begin{proof}
    See Appendix \ref{proof:forward_invariance_phi_1}.
\end{proof}

Theorem \ref{thm:forward_invariance_phi_1} implies that the argument of $\phi_1$ remains in its admissible domain for all $k \ge 0$. 
Consequently, the admissible domain of $\phi_1$ is forward invariant along the closed-loop trajectories.

\subsection{Forward Invariance of the Admissible Domain of $\phi_2$}

In contrast to the first backstepping step, enforcing the deadbeat choice $\rho_{2_k}=0$ from the initial time may be overly restrictive, since the resulting control input must satisfy not only the admissibility condition of the lifting function $\phi_2$, but also the admissibility constraints previously imposed for $\phi_1$. 
To alleviate this restriction, a transient state-dependent gain $\rho_{2_k}$ with $|\rho_{2_k}|<1$ is employed so that the admissibility condition of $\phi_2$ is preserved while the trajectory evolves toward a region where the deadbeat choice becomes admissible. 
The following result establishes forward invariance of the admissible domain of $\phi_2$ during this transient phase.

\begin{theorem}[Forward Invariance of the Admissible Domain of $\phi_2$]
\label{thm:forward_invariance_phi_2}

Consider the transformed dynamics \eqref{eq:z2_dot} with the control law \eqref{eq:u_k_constraint_lifting}. 
Assume that Theorem \ref{thm:forward_invariance_phi_1} is satisfied. 
Suppose that, for all $0 \le k < k_\rms$, the gain $\rho_{2_k}$ satisfies $|\rho_{2_k}| < 1$ and 
\begin{align}
    \left|
        \rho_{2_k} e_{2_k}
        +
        \psi_2\!\left(
            \zeta_{2\rmd_{k+1}}   
        \right)
    \right|
    &< 1.
    \label{eq:phi2_admissibility_transient_theorem}
\end{align}
Then, for all $0 \le k < k_\rms$, where $k_\rms$ is the finite switching time, the argument of $\phi_2$ remains admissible. 
Hence, the admissible domain of $\phi_2$ is forward invariant for all $0 \le k < k_\rms.$
\end{theorem}

\begin{proof}
    See Appendix \ref{proof:forward_invariance_phi_2}.
\end{proof}

Theorem \ref{thm:forward_invariance_phi_2} implies that the argument of $\phi_2$ remains in its admissible domain throughout the transient phase $0 \le k < k_\rms$. 
Consequently, the admissible domain of $\phi_2$ is forward invariant along the closed-loop trajectories during this phase.

\subsection{Conditional Deadbeat Property of the Second Backstepping Step}
The transient phase described above ensures that the argument of the lifting function $\phi_2$ remains admissible while the trajectory evolves toward a region where the deadbeat choice becomes feasible. 
Once such a region is reached, the controller can be switched to the deadbeat choice $\rho_{2_k}=0$. 
The following result establishes that, provided the admissibility condition remains satisfied at the switching time, the second backstepping step operates in deadbeat mode thereafter while preserving admissibility of $\phi_2$.

\begin{theorem}[Conditional Deadbeat Property of the Second Backstepping Step]
\label{thm:conditional_deadbeat_phi2}

Consider the transformed dynamics \eqref{eq:z2_dot} with the control law \eqref{eq:u_k_constraint_lifting}. 
Assume that Theorem \ref{thm:forward_invariance_phi_2} is satisfied.
If, at $k=k_\rms$, the deadbeat choice $\rho_{2_k}=0$ satisfies
\begin{align}
    \left|
        \psi_2\!\left(
            \zeta_{2\rmd_{k_\rms+1}}   
        \right)
    \right|
    < 1,
    \label{eq:psi2condition}
\end{align}
and, for all $k \ge k_\rms,$ $\rho_{2_k}=0,$
then, for all $k \ge k_\rms,$
    $e_{2_{k+1}} = 0,$
and the argument of $\phi_2$ remains admissible for all $k \ge k_\rms$. 
\end{theorem}

\begin{proof}
    See Appendix \ref{proof:conditional_deadbeat_phi2}.
\end{proof}

Theorem \ref{thm:conditional_deadbeat_phi2} implies that, once the switching condition \eqref{eq:psi2condition} is satisfied, the second-step error is driven to zero in one step and remains identically zero thereafter. 
Consequently, the admissibility of the argument of $\phi_2$ is preserved for all $k \ge k_\rms$.

\section{Numerical Example}

This section illustrates the controller developed in Section \ref{sec:Controller_design} and verifies the admissibility conditions for $\phi_1$ and $\phi_2$ established in Section \ref{sec:Admisibility_constraint_leifting}. 
Consider the discrete-time double integrator
\begin{align}
    x_{1_{k+1}} &= x_{1_k} + x_{2_k},
    \label{eq:x1_double_integrator}
    \\
    x_{2_{k+1}} &= x_{2_k} + u_k.
    \label{eq:x2_double_integrator}
\end{align}
The objective is to follow the command $x_{1\rmd_k}$ while simultaneously satisfying position and velocity constraints, $|x_{1_k}| < \overline{x}_1,$ and $|x_{2_k}| < \overline{x}_2.$

In this example, the inverse sigmoid lifting function $\phi(\chi) = \mathrm{atanh}(\chi)$ is used for both states. 
Furthermore, we set $\rho_1 = 0$. 
Then,
\begin{align}
    \SF_1(z_{1_k},z_{2_k})
        &=
            \dfrac{1}{\overline{x}_1}
            \left(
                x_{1_k}
                +
                x_{2_k}
            \right)
        \nn \\
        &=
        \tanh
                \left(
                \zeta_{1_k}
                \right)
                +
                    \dfrac{\overline{x}_2}{\overline{x}_1}
                \tanh
                (\zeta_{2_k})
    \\
    \SF_2(z_{1_k},z_{2_k},u_{k})
        &=
        \dfrac{1}{\overline{x}_2}
        \left(
            \rho_{2_k}
            -
            1
        \right)
        \left(
            x_{2_k}
            +
            x_{1_k}
            -
            x_{1\rmd_{k}}
        \right)
        \nn \\
        &=
        \tanh (\zeta_{2_k})
                +
                \dfrac
                {
                    u_k
                }
                {
                    \overline{x}_2
                }
    \label{eq:viable_set_z2_dot_closed_loop}
\end{align}

The initial conditions are selected inside the safe set $\SSS$.
Furthermore, the admissibility conditions associated with the sets $\SA_1$ and $\SA_2$ impose the additional constraints on the initial conditions, given by
\begin{align}
    |x_{1_0} + x_{2_0}| 
        &< 
            \overline{x}_1,
    \quad
    |x_{1_0} + x_{2_0} - x_{1\rmd_k}|
        < 
            \overline{x}_2,
    \label{eq:cond_DI_Pos_Vel}
\end{align}
To simplify the controller implementation, we set $\rho_{2_0}=0$ in \eqref{eq:viable_set_z2_dot_closed_loop}, which yields the condition \eqref{eq:cond_DI_Pos_Vel}. 
Next, we choose
\begin{align}
    \rho_{2_k}
        =
            \left\{
            \begin{array}{cc}
                 0, &  |\Delta x_k| < \overline{x}_2,
                 \\[4pt]
                 1 - \dfrac{\overline{x}_2}{2|\Delta x_k|}, & |\Delta x_k| \geq \overline{x}_2,
            \end{array}
            \right.
    \label{eq:rho_2_cond_double_integrator}
\end{align}
where
    $\Delta x_k \isdef x_{2_k} + x_{1_k} - x_{1\rmd_k}.$
Note that this switching law is not unique, it is one admissible choice consistent with the analysis of Section \ref{sec:Admisibility_constraint_leifting}.
Finally, the controller $u_k$ is given by \eqref{eq:u_k_constraint_lifting}.


Figure \ref{fig:Pos_Vel_Const_Rand_a_great_n_K2_R1} shows the closed-loop trajectories and phase portraits for three randomly generated initial conditions and corresponding randomly generated commands. 
The admissible sets $\SA_1$ and $\SA_2$ are shown in shaded blue and shaded green, respectively. 
The initial state is shown with a red square and the converged state is shown in green triangle. 
Note that $\SA_1$ does not change, whereas $\SA_2$ changes as the gain $\rho_{2_k}$ varies according to the switching law in \eqref{eq:rho_2_cond_double_integrator}. 
Despite this evolution, the closed-loop trajectories remain entirely within the admissible sets $\SA_1$ and $\SA)2$ and the safe set $\SSS$ and converge to the commanded value.


\begin{figure}[!ht]
    \centering
    \includegraphics[width=\columnwidth]{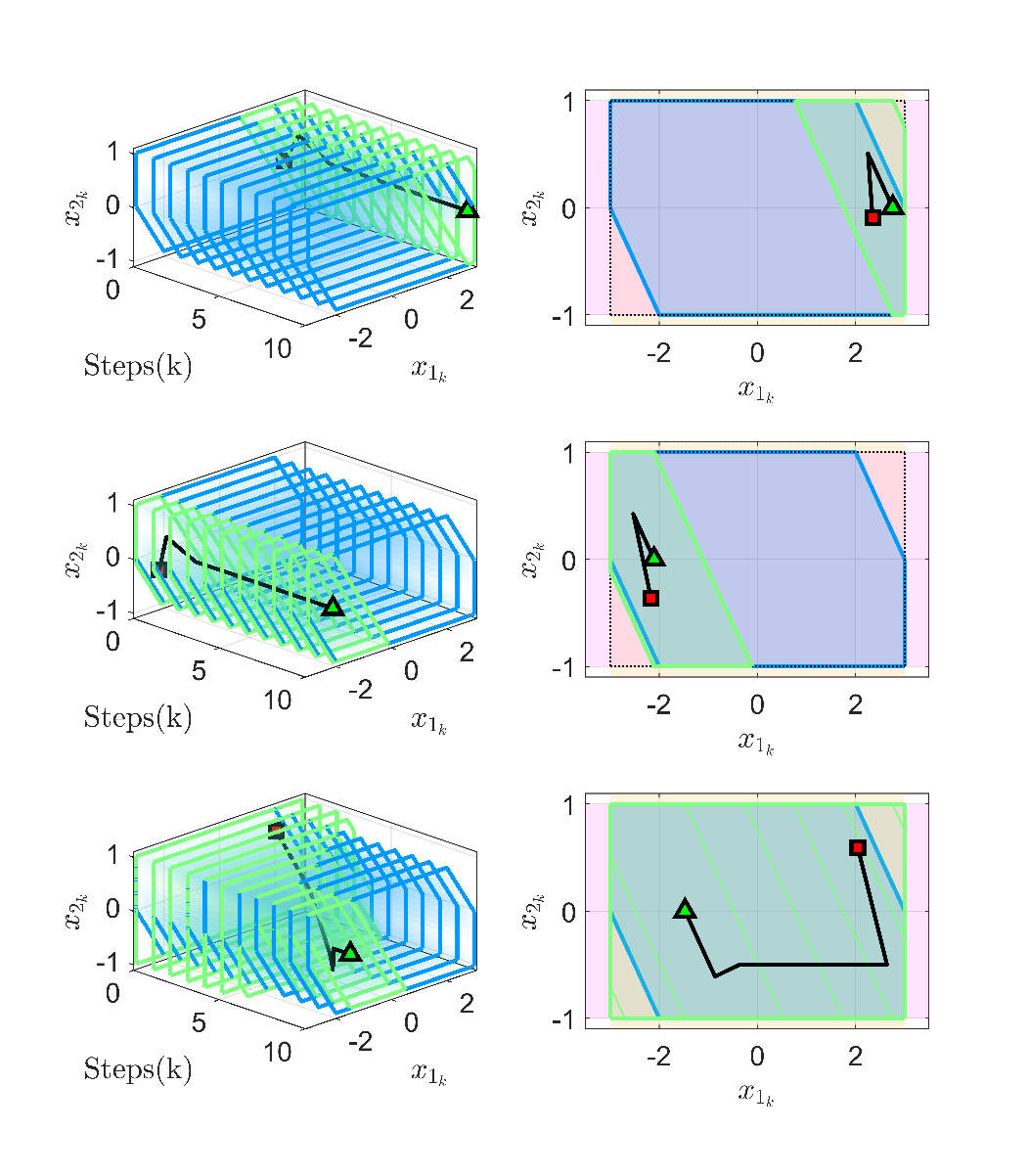}
    \caption{Closed-loop trajectory and phase portrait for the position- and velocity-constrained double integrator with randomly generated constant commands. 
    The admissible sets $\SA_1$ and $\SA_2$ are shown in shaded blue and shaded green, respectively.
    At each step, the boundaries of $\SA_1$ and $\SA_2$ are shown in solid blue and solid green, respectively. 
    Note that $\SA_1$ remains unchanged, whereas $\SA_2$ changes as the gain $\rho_{2_k}$ varies.
    }
    \label{fig:Pos_Vel_Const_Rand_a_great_n_K2_R1}
\end{figure}

Figure \ref{fig:Pos_Vel_Const_Rand_a_great_n_K2_R1_states} shows the corresponding time evolution of the states, control input, and gain $\rho_{2_k}.$ 
The safe set $\SSS$ for the states $x_1$ and $x_2$ is shown in shaded yellow and pink, respectively.
The gain $\rho_{2_k}$ varies only during the transient phase to maintain admissibility of the lifting dynamics and eventually is set to $0$ once the state $x_1$ converges to the commanded value. 
\begin{figure}[!ht]
    \centering
    \includegraphics[width=\columnwidth]{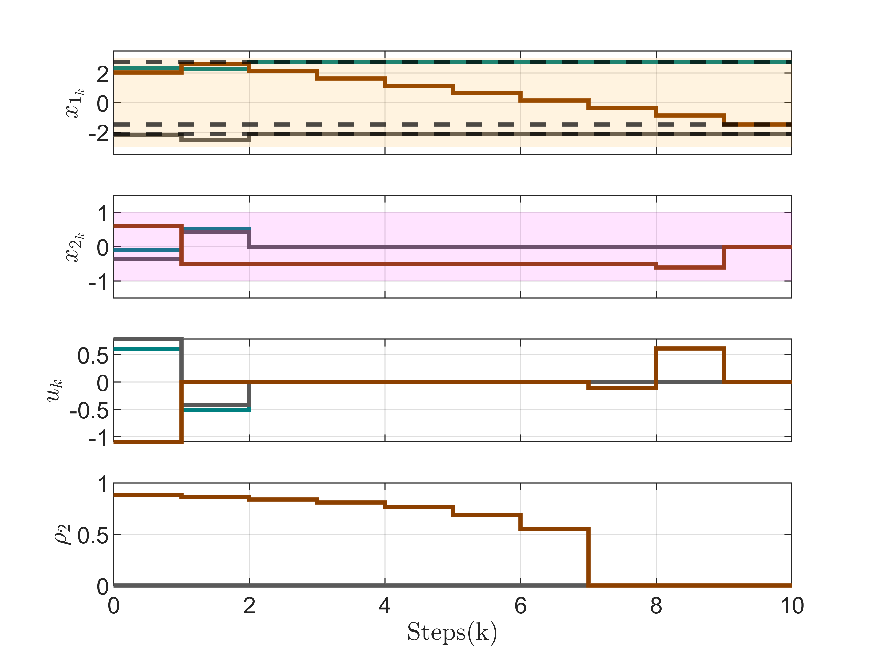}
    \caption{State, input, and gain evolution for the position- and velocity-constrained double integrator with randomly generated constant commands. 
    The safe set $\SSS$ for the states $x_1$ and $x_2$ is shown in shaded yellow and pink, respectively.
    The gain $\rho_{2_k}$ varies during the transient phase to preserve admissibility of the lifting dynamics.
    }
    \label{fig:Pos_Vel_Const_Rand_a_great_n_K2_R1_states}
\end{figure}

Figure \ref{fig:Pos_Vel_Const_Rand_K2_a_great_b_track_R1} shows the closed-loop trajectories and phase portraits for three randomly generated initial conditions under a harmonic command $x_{1\rmd_k} = 0.5 \sin (0.5 k).$ 
%
The admissible sets $\SA_1$ and $\SA_2$ are shown in shaded blue and shaded green, respectively.
Note that $\SA_1$ does not change, whereas $\SA_2$ changes as the gain $\rho_{2_k}$ varies according to the switching law in \eqref{eq:rho_2_cond_double_integrator}. 


\begin{figure}[!ht]
    \centering
    \includegraphics[width=\columnwidth]{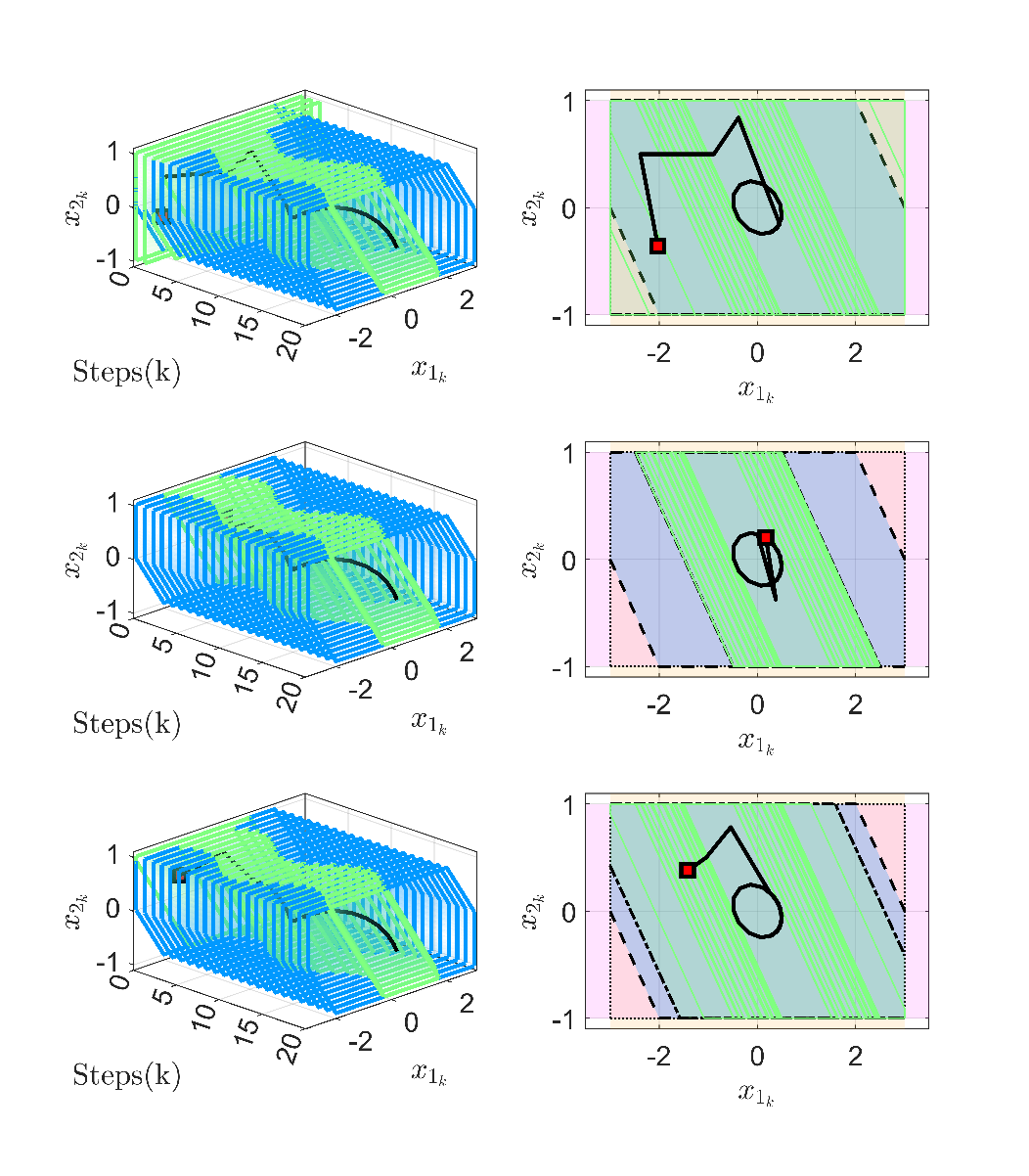}
    \caption{Closed-loop phase portrait for the position- and velocity-constrained double integrator with a sinusoidal command. 
    The admissible sets $\SA_1$ and $\SA_2$ are shown in shaded blue and shaded green, respectively.
    At each step, the boundaries of $\SA_1$ and $\SA_2$ are shown in solid blue and solid green, respectively. 
    Note that $\SA_1$ remains unchanged, whereas $\SA_2$ changes as the gain $\rho_{2_k}$ varies.
    }
    \label{fig:Pos_Vel_Const_Rand_K2_a_great_b_track_R1}
\end{figure}

Figure \ref{fig:Pos_Vel_Const_Rand_K2_a_great_b_track_R1_states} shows the corresponding time evolution of the states, control input, and gain $\rho_{2_k}.$ 
The safe set $\SSS$ for the states $x_1$ and $x_2$ is shown in shaded yellow and pink, respectively.
The gain $\rho_{2_k}$ varies only during the transient phase to maintain admissibility of the lifting dynamics and eventually is set to $0$ once the state $x_1$ converges to the commanded value. 

\begin{figure}[!ht]
    \centering
    \includegraphics[width=\columnwidth]{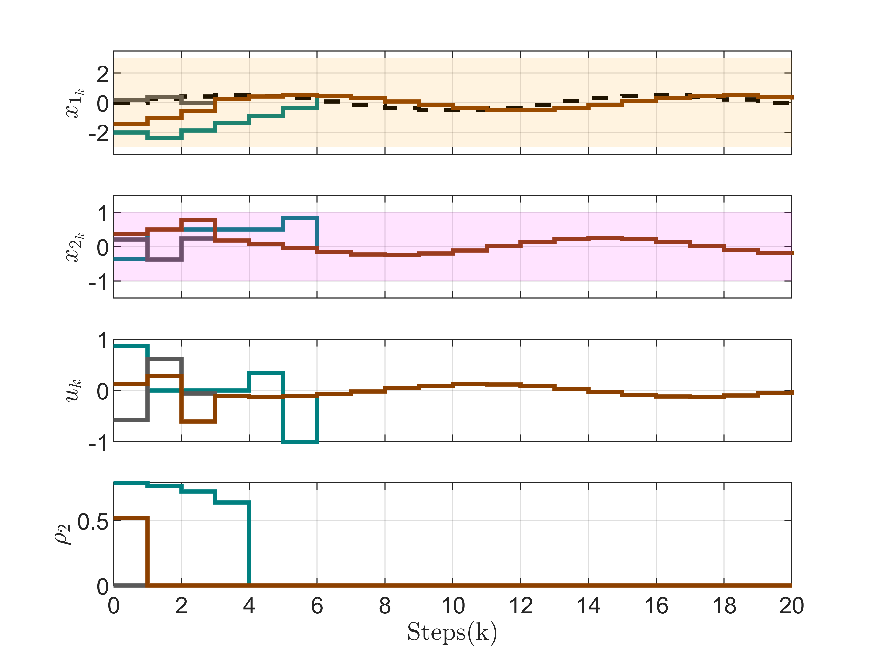}
    \caption{State, input, and gain evolution for the position- and velocity-constrained double integrator with a sinusoidal command. 
    The safe set $\SSS$ for the states $x_1$ and $x_2$ is shown in shaded yellow and pink, respectively.
    The gain $\rho_{2_k}$ varies during the transient phase to preserve admissibility of the lifting dynamics.}
    \label{fig:Pos_Vel_Const_Rand_K2_a_great_b_track_R1_states}
\end{figure}

\section{Conclusions}

This paper developed a constraint-lifting control framework for discrete-time strict-feedback nonlinear systems subject to state constraints. 
The proposed approach transforms the constrained control problem into an equivalent unconstrained one using smooth constraint-lifting mappings constructed from strictly increasing sigmoid functions. 
Controller synthesis is performed in the lifted coordinates using a recursive backstepping design, while the inverse transformation guarantees that the original state constraints are satisfied.

Sufficient conditions on the controller gains were derived to ensure asymptotic stability of the closed-loop system and forward invariance of the admissible domain of the lifting functions. 
The analysis further established a conditional deadbeat property for the second backstepping step once the closed-loop trajectory enters a region where the admissibility conditions are satisfied. 
Numerical simulations on a constrained double-integrator system demonstrated that the proposed method preserves state constraints while achieving stable regulation and tracking.

Future work will investigate extensions of the proposed framework to vector-valued constraints and adaptive control settings.

\label{sec:conclusions}

\appendix
\section{Proofs}
\subsection{Proof of Theorem \ref{thm:forward_invariance_phi_1}}
\label{proof:forward_invariance_phi_1}
\begin{proof}

From \eqref{eq:z1_dot}, the lifted dynamics is
    $z_{1_{k+1}}
        =
            \overline{x}_1
            \phi_1
            \bigl(
                \SF_1(z_{1_k},z_{2_k})
            \bigr).$
Hence, $z_{1_{k+1}}$ is well defined if \eqref{eq:phi1_domain_requirement} holds.

From the tracking error definition \eqref{eq:e1_def}, evaluated at $k+1$,
\begin{align}
    e_{1_{k+1}}
        &=
            \psi_1
            \left(
                \zeta_{1_{k+1}}
            \right)
            -
            \psi_1
            \left(
                \zeta_{1\rmd_{k+1}}  
            \right)
        \nn \\
        &=
            \psi_1
            \Bigl(
                \phi_1
                \bigl(
                    \SF_1(z_{1_k},z_{2_k})
                \bigr)
            \Bigr)
            -
            \psi_1
            \left(
                \zeta_{1\rmd_{k+1}}  
            \right)
    \nn
    \\
    &=
            \SF_1(z_{1_k},z_{2_k})
            -
            \psi_1
            \left(
                \zeta_{1\rmd_{k+1}}  
            \right),
        \nn
\end{align}
since $\psi_1$ is the inverse of $\phi_1$ on the admissible domain.
Imposing the first contraction condition \eqref{eq:first_BS_Step} with $\rho_1=0$ yields
\begin{align}
    \SF_1(z_{1_k},z_{2_k})
        =
            \psi_1
            \left(
                \zeta_{1\rmd_{k+1}}  
            \right).
\label{eq:SF1_equals_psi}
\end{align}
Next, it follows from the fact $\psi_1=\phi_1^{-1}$ and \eqref{eq:z1d} that 
    $\psi_1
    \left(
        \zeta_{1\rmd_{k+1}}  
    \right)
        =
            \chi_{1\rmd_{k+1}} .$
Substituting into \eqref{eq:SF1_equals_psi} yields
    $\SF_1(z_{1_k},z_{2_k})
        =
            \chi_{1\rmd_{k+1}} .$

Finally, since the desired trajectory satisfies 
$|x_{1\rmd_{k+1}}| < \overline{x}_1 ,$
it follows that 
$|\chi_{1\rmd_{k+1}}| < 1,$
which proves \eqref{eq:phi1_domain_requirement}.
\end{proof}

\subsection{Proof of Theorem \ref{thm:forward_invariance_phi_2}}
\label{proof:forward_invariance_phi_2}
\begin{proof}
It follows from \eqref{eq:second_BS_Step} that 
\begin{align}
    \SF_2(z_{1_k},z_{2_k},u_k)
        &=
            \rho_{2_k} e_{2_k}
            +
            \psi_2
            \left(
                \zeta_{2\rmd_{k+1}}   
            \right). \nn
    \label{eq:SF2_relation_theorem}
\end{align}
Hence, \eqref{eq:phi2_admissibility_transient_theorem} implies, for all $0 \leq k < k_\rms,$
    $\left|
        \SF_2(z_{1_k},z_{2_k},u_k)
    \right|
        <
            1.$ 
%
Therefore, the argument of $\phi_2$ remains within its admissible domain for all
$0 \leq k < k_\rms$.

Moreover, the second-step error dynamics satisfies \eqref{eq:e2_dynamics_theorem}.
Since $|\rho_{2_k}|<1$ for all $0 \leq k < k_\rms$, the error dynamics are
contractive during the transient phase. Hence, the admissibility condition of
$\phi_2$ is preserved throughout the transient phase.
\end{proof}

\subsection{Proof of Theorem \ref{thm:conditional_deadbeat_phi2}}
\label{proof:conditional_deadbeat_phi2}

\begin{proof}

Assume that the conditions of Theorem
\ref{thm:forward_invariance_phi_2} hold for all $0 \leq k < k_\rms$.
Suppose that, at time $k_\rms$, the closed-loop trajectory reaches a point for
which the deadbeat choice $\rho_{2_k}=0$ preserves the admissibility of the
argument of $\phi_2$.
Setting $\rho_{2_k}=0$ into \eqref{eq:second_BS_Step} yields
    $\SF_2(z_{1_k},z_{2_k},u_k)
        =
            \psi_2
            \left(
                \zeta_{2\rmd_{k+1}}   
            \right).$
Hence, by \eqref{eq:psi2condition}, the admissibility of the argument of $\phi_2$ is preserved for all $k \geq k_\rms$.

Furthermore, substituting $\rho_{2_k}=0$ into the error dynamics
\eqref{eq:e2_dynamics_theorem} implies that, for all, $k \geq k_\rms,$
    $e_{2_{k+1}}
        =
            0.$
Therefore, once the switching condition is satisfied, the second-step error is
driven to zero in one step and remains identically zero thereafter. This
establishes the deadbeat property of the second backstepping step.
\end{proof}

\printbibliography
\end{document}